\documentclass[journal]{IEEEtran}
\usepackage{amssymb}
\usepackage{framed}
\usepackage{bm}
\usepackage{cite}
\usepackage{comment}
\usepackage{algorithm}
\usepackage{algorithmic}
\usepackage{enumitem}
\usepackage{amsmath}
\usepackage{color}
\newtheorem{theorem}{Theorem}

\newtheorem{remark}{Remark}

\newtheorem{proposition}{Proposition}

\newtheorem{proof}{Proof}

\usepackage{blkarray}

\newcommand{\maximize}{\mathop{\rm maximize}\limits}

\newcommand{\Real}{\mathbb{R}}

\DeclareMathOperator{\rank}{rank}

\usepackage{latexsym}
\def\qed{\hfill $\Box$} 

\usepackage{ascmac}

\usepackage{cite}

\ifCLASSINFOpdf
  \usepackage[pdftex]{graphicx}
\else
  \usepackage[dvips]{graphicx}
\fi
\usepackage{amsmath}
\ifCLASSOPTIONcompsoc
  \usepackage[caption=false,font=normalsize,labelfont=sf,textfont=sf]{subfig}
\else
  \usepackage[caption=false,font=footnotesize]{subfig}
\fi
\usepackage{url}

\begin{document}
%
\title{Minimal controllability problems on linear structural descriptor systems}
\author{Shun Terasaki and Kazuhiro~Sato \thanks{S. Terasaki and K. Sato are with the Department of Mathematical Informatics, Graduate School of Information Science and Technology, The University of Tokyo, Tokyo 113-8656, Japan, email:terasaki@g.ecc.u-tokyo.ac.jp (S. Terasaki), kazuhiro@mist.i.u-tokyo.ac.jp (K. Sato) }}
\maketitle
\thispagestyle{empty}
\pagestyle{empty}

\begin{abstract}
We consider minimal controllability problems (MCPs) on linear structural descriptor systems.
We address two problems of determining the minimum number of input nodes such that a descriptor system is structurally controllable.
We show that MCP0 for structural descriptor systems can be solved in polynomial time. 
This is the same as the existing results on a typical structural linear time invariant (LTI) systems.
However, the derivation of the result is considerably different because the derivation technique of the existing result cannot be used for descriptor systems. Instead, we use the Dulmage--Mendelsohn decomposition.
Moreover, we prove that the results for MCP1 are different from those for usual LTI systems. 
In fact, MCP1 for descriptor systems is an NP-hard problem, while MCP1 for LTI systems can be solved in polynomial time. 
\end{abstract}

\begin{IEEEkeywords}
structural controllability, large-scale system, descriptor system, DM decomposition
\end{IEEEkeywords}

%
\IEEEpeerreviewmaketitle

\section{Introduction} \label{sec:intro}
%
%
%

Controllability, which was introduced by \cite{kalman1960general}, for network dynamical systems such as multi-agent systems\cite{mesbahi2010graph}, brain networks\cite{gu2015controllability}, and power networks\cite{pagani2013power} has received considerable research attention.
In the simplest case, systems can be modeled as linear time invariant (LTI) systems with state $x(t) \in \Real^n$ and input $u(t) \in \Real^m$,
\begin{align}
    \dot{x}(t) = Ax(t) + Bu(t), \label{eq:lti}
\end{align}
where $A \in \Real^{n\times n}$ and $B \in \Real^{n\times m}$ are large constant matrices.
In control theory literature, qualitative and quantitative controllability problems of designing $B$ under some constraints have been considered.
The qualitative problem includes a problem that makes system (\ref{eq:lti}) controllable with minimum input nodes \cite{liu2011controllability,olshevsky2015minimal,pequito2015framework,olshevsky2015minimum} by employing a structural control theory developed in \cite{lin1974structural,shields1976structural}.
In contrast, the quantitative problem includes the problem of maximizing some metrics of the system; in particular, the controllability Gramian \cite{summers2015submodularity,pasqualetti2014controllability,sato2020controllability,yan2015spectrum}.
In a real-world system, each parameter in $A$ and $B$ is usually not precisely determined; further, the quantitative problem often becomes computationally difficult as the state dimension $n$ becomes larger. That is, we can only handle up to a few hundred state dimensions in practice. Thus, considering the qualitative problem is more suitable for large-scale networked systems. 

To address the controllability problem of large-scale networked systems, we consider the qualitative problems called minimal controllability problems (MCPs)\cite{olshevsky2015minimal}.
The MCP0 is a problem that finds an $n\times m$ matrix $B$ such that system (\ref{eq:lti}) is structurally controllable and $m$ is minimum. Here, system (\ref{eq:lti}) is called structurally controllable if system (\ref{eq:lti}) becomes controllable by setting nonzero values in $A$ and $B$ \cite{lin1974structural}.
The MCP1 is a problem that finds an $n\times n$ diagonal matrix $B$ such that system (\ref{eq:lti}) is structurally controllable and the number of nonzero elements in $B$ is minimum.
Although both MCP0 and MCP1 can be seen as a problem of minimizing the number of input nodes, each input in the MCP1 case is not allowed to connect to multiple nodes (the visual difference is shown in Fig.~\ref{fig:comparisonMCP01}). 
Polynomial time algorithms for MCP0 and MCP1 associated with system (\ref{eq:lti}) were proposed in \cite{liu2011controllability,pequito2015framework,olshevsky2015minimum}, respectively.

In this paper, we extend the results of \cite{liu2011controllability,pequito2015framework,olshevsky2015minimum} on MCPs for a descriptor system
\begin{align}
    F\dot{x}(t) = Ax(t) + Bu(t),  \label{eq:descriptor}
\end{align}
which is a generalized system of system (\ref{eq:lti}), where $F\in \Real^{n\times n}$ can be a singular matrix. The descriptor formulation is well suited to model practical systems with algebraic constraints such as electric circuit systems \cite{dai1989singular,murota00,duan2010analysis}.
The structural controllability of descriptor system (\ref{eq:descriptor}) is investigated by \cite{yamada1985generic,yip1981solvability}; however, there is less discussion in terms of the input selection problem.
In fact, to the best of our knowledge, the input selection problem based on structural controllability for descriptor system (\ref{eq:descriptor}) is only found in \cite{clark2017input}. However, \cite{clark2017input} only deals with a special case in (\ref{eq:descriptor}), as explained in Section~\ref{subsec:resultdesc}.

\begin{figure}[t]
    \centering
    \begin{tabular}{c}
    \begin{minipage}{0.5\hsize}
    \centering
    \includegraphics[width=3.0cm]{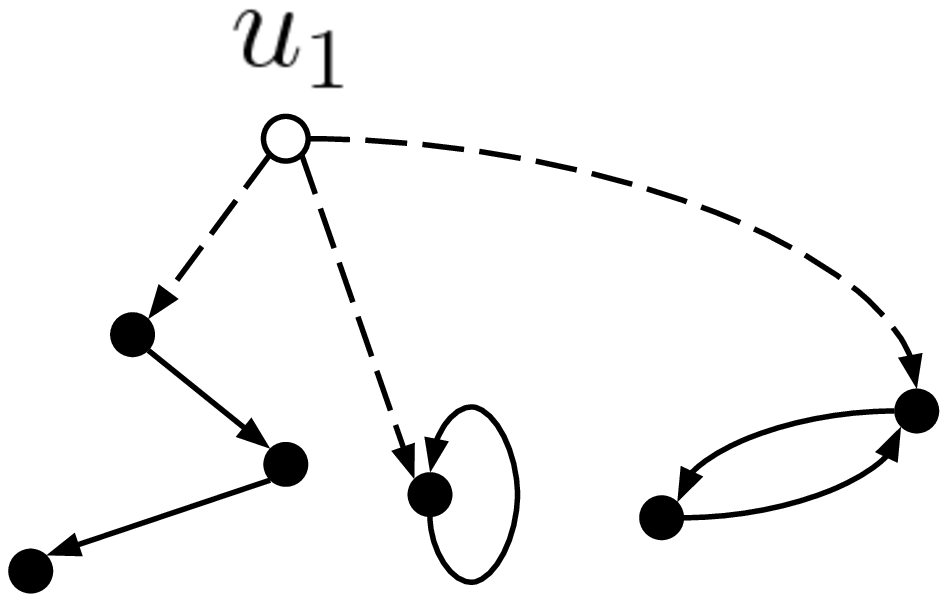}
    \hspace{1.8cm} (a) MCP0 
    \end{minipage}
    \begin{minipage}{0.5\hsize}
    \centering
    \includegraphics[width=3.0cm]{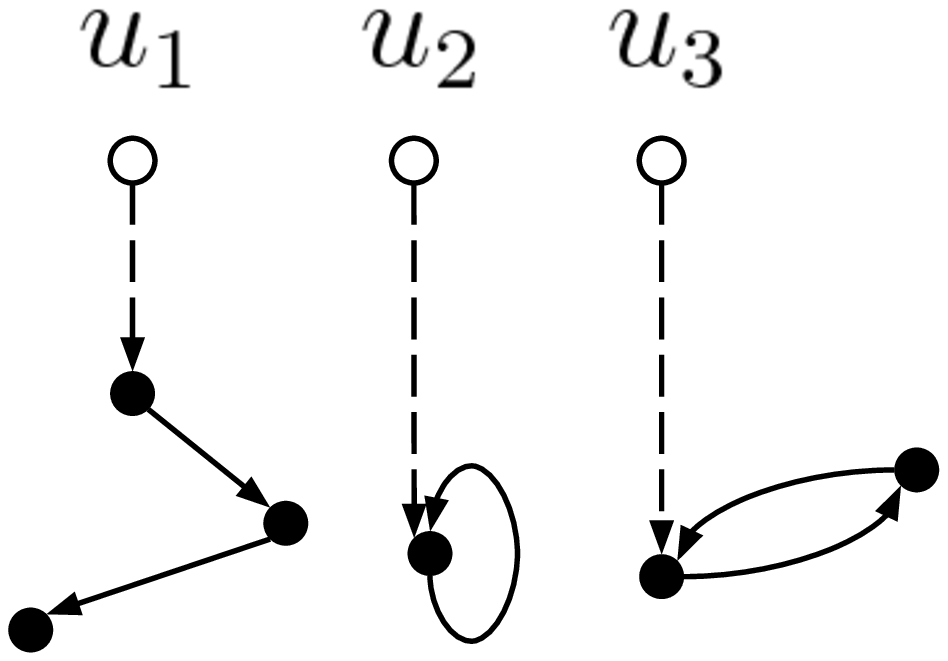}
    \hspace{1.8cm} (b) MCP1
    \end{minipage}
    \end{tabular}
    \caption{Difference between MCP0 and MCP1.}
    \label{fig:comparisonMCP01}
\end{figure}

The contributions of this study can be summarized as follows.
\begin{itemize}
    \item We provide the solution of MCP0 for structural descriptor system (\ref{eq:descriptor}). The solution shows that the minimum number of input nodes is determined by maximum matching. This is the same as the result in \cite{liu2011controllability}.
    However, the derivation is different from that in \cite{liu2011controllability} based on graph concepts such as cacti. This is because a directed graph representation cannot be applied to descriptor system (\ref{eq:descriptor}) unlike LTI system (\ref{eq:lti}), as explained in Section~\ref{subsec:resultlti}. Thus, instead of the graph concepts used in \cite{liu2011controllability}, we use a bipartite graph representation and the Dulmage--Mendelsohn (DM) decomposition  \cite{dulmage1963two}\cite{murota00} for descriptor system (\ref{eq:descriptor}).
    \item We prove the computational hardness of MCP1 for structural descriptor system (\ref{eq:descriptor}). That is, we show that the time complexity belongs to the NP-hard class by reducing the set cover problem, which is an NP-hard problem, to MCP1 for an instance of structural descriptor system (\ref{eq:descriptor}). The result is completely different from that for system (\ref{eq:lti}), which can be solved in polynomial time.
\end{itemize}

The remainder of this paper is organized as follows.
Section~\ref{sec:prevresults} describes the results of previous studies on MCPs.
The formulation of the novel MCPs for structural descriptor systems is described in Section~\ref{sec:prb}.
In Section~\ref{sec:analysis}, we provide the analysis of the MCPs for system (\ref{eq:descriptor}). The conclusions are presented in Section~\ref{sec:conclusion}.

\section{Previous results of the MCPs}
\label{sec:prevresults}
This section describes the existing MCPs results to highlight the difference with our study.
\subsection{Basic concepts of graph theory}
\label{subsec:graph}
We summarize the basic concepts of graph theory that are used in this study.

A strongly connected component (SCC) for a directed graph with nodes set $V$ is a maximal subset $C\subseteq V$ whose nodes $u,v \in C$ can be connected by a directed path on the graph (Fig.~\ref{fig:directedgraph}).
A directed matching of a directed graph with edges set $E$ is a subset $M\subseteq E$ that does not share common start or end nodes. Directed matching with the maximum size is called maximum directed matching.

For a bipartite graph $G = (V^+,V^-;E)$, $\partial^+ e$ and $\partial^- e$ with edge $e = (v^+,v^-) \in E$ denote the nodes of $V^+$ and $V^-$, respectively. Edges set $M\subseteq E$ is matching if it does not share the end of each edge. A matching $M$ is called maximum matching if $M$ contains the largest possible number of edges (Fig.~\ref{fig:bipartite}). If all nodes are ends of some edges in a maximum matching $M$, $M$ is called perfect.
For a matching $M$ of $G$, the elements of $V^+\setminus \partial^+M$ and $V^-\setminus \partial^-M$ are called right-unmatched nodes and left-unmatched nodes, respectively.

\begin{figure}
    \centering
    \includegraphics[width=3.5cm]{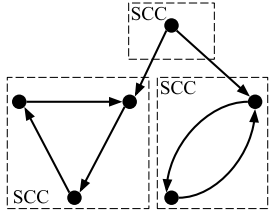}
    \caption{Directed graph and its SCCs.}
    \label{fig:directedgraph}
\end{figure}
\begin{figure}
    \centering
    \includegraphics[width=4cm]{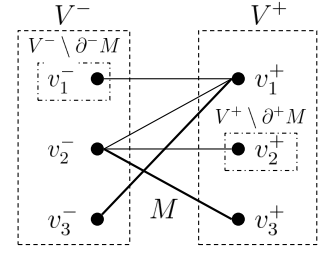}
    \caption{Bipartite graph. Bold edges represent an example of the maximum matching $M$ of the graph. The sets of left and right unmatched nodes are $V^- \setminus \partial^- M = \{v_1^-$\} and $V^+\setminus \partial^+ M = \{v_2^+\}$, respectively}
    \label{fig:bipartite}
\end{figure}
\subsection{Results of MCPs on system (\ref{eq:lti})}
\label{subsec:resultlti}
We first introduce the results of MCPs for structural system (\ref{eq:lti}).
MCP0 and MCP1 as mentioned in Section~\ref{sec:intro} can be formulated as
\begin{align*}
& \text{(MCP0)} 
\begin{cases}
    \underset{\text{$B\in \mathcal{G}^{n\times m}$}}{\text{minimize}} & m \\
    \text{subject to} & \text{system (\ref{eq:lti}) is} \\ & \text{structurally controllable,}
\end{cases}\\
& \text{(MCP1)}
    \begin{cases}
        \underset{\text{$b\in \mathcal{G}^{n}$}}{\text{minimize}} & \|b\|_0 \\
        \text{subject to} & \text{system (\ref{eq:lti}) with $B = \mathrm{diag}(b)$ is} \\ & \text{structurally controllable,}
    \end{cases}
\end{align*}
where $\|\cdot\|_0$ denotes the number of nonzero elements of a matrix, and $\mathcal{G}^{n\times m}$ and $\mathcal{G}^{n}$ denote the set of all $n\times m$ generic matrices and $n$ generic vectors, respectively.
A matrix is called generic if each nonzero element is an independent parameter. For a more precise definition of the generic matrix, see \cite{murota00}.

To provide the solution to MCPs, we define a directed graph representation $G(A,B) = (V,E)$ of system (\ref{eq:lti}), which is the fundamental tool for structural system analysis.
The nodes set $V = X \cup U$ with
\begin{align}
    X & = \{x_1,\dots,x_n\}, \label{eq:nodex}\\
    U & = \{u_1,\dots,u_m\}, \label{eq:nodeu}
\end{align}
corresponds to the state and input variables, and an edge set is defined as $E = E_A \cup E_B$ with $E_A = \{(x_j,x_i)\mid a_{ij}\ne 0\}$ and $E_B = \{(x_j,u_i)\mid b_{ij} \ne 0 \}$. 
Here, $E_A$ represents the connections between state nodes and $E_B$ represents connections between the state and input nodes.
A digraph representation of the autonomous system (\ref{eq:lti}) with $u(t) = 0$ is defined as $G(A) := (X,E_A)$.

The minimum number of input nodes $n_D$, \textit{i.e.,} the optimal value of MCP0, is obtained by finding the maximum directed matching of $G(A)$ \cite{liu2011controllability}
\begin{align*}
    n_D = \max \{ n - |M^\ast| , 1 \},
\end{align*}
where $M^\ast$ denotes the maximum directed matching of $G(A)$.
The maximum directed matching $M^\ast$ of $G(A)$ can be calculated as the maximum bipartite matching by translating $G(A)$ to a corresponding bipartite graph $B(A) = (V^+,V^-,E)$. The node sets $V^+$ and $V^-$ are defined as $V^+ := X^+\cup U$ and $V^- := X^-$ with $X^+ := \{x_1^+,\dots,x_n^+\}$ and $X^- := \{x_1^-,\dots,x_n^-\}$. The elements of $X^+$ and $X^-$ correspond to the column and row numbers of $A$, respectively.
In addition, the edge set $E$ is defined as $E_A \cup E_B$, where $E_A = \{(x_j^+,x_i^-)\mid a_{ij}\ne 0\}$ and $E_B = \{(u_j,x_i^-)\mid b_{ij} \ne 0 \}$.
Then, the maximum directed matching $M'$ of $G(A)$ can be regarded as the maximum matching $M$ of $B(A)$ by identifying $V^+$ and $V^-$ of $B(A)$. 
Moreover, the minimum input nodes in MCP0 correspond to the left-unmatched nodes of a maximum matching of $B(A)$.

To describe the solution to MCP1, we note that the optimal value $p$ of MCP1 is given as $p = m + \beta - \alpha$,
which was shown in \cite{pequito2015framework}, where $m$ denotes the number of left-unmatched nodes in any maximum matching of the bipartite graph $B(A)$, and $\beta$ denotes the number of root SCCs, that have no incoming links from other SCCs.
A top assignable SCC is a root SCC that contains at least one left-unmatched node associated with a maximum matching $M^\ast$, and $\alpha$ is the maximum number of top assignable SCCs.
This maximum top assignability index can be calculated using the minimum weight maximum matching whose time complexity is $O(|V|^3)$. For more details, see \cite{korte2012combinatorial}. A faster method for solving MCP1 using the minimum cost allowed matching is proposed by \cite{olshevsky2015minimum}; its time complexity is $O(|V|+|E||V|^{1/2})$.

\subsection{Results of MCPs on system (\ref{eq:descriptor})}
\label{subsec:resultdesc}
The input selection problems that make descriptor system (\ref{eq:descriptor}) structurally controllable are proposed in \cite{clark2017input}. In \cite{clark2017input}, the authors studied the following special case of system (\ref{eq:descriptor}) and the problem of minimizing the size of $S$.
\begin{align}
\begin{bmatrix}
    \hat{F} \\
    0
\end{bmatrix} \begin{bmatrix}
 \dot{x}_R(t) \\
 \dot{x}_S(t)
\end{bmatrix}
& = \begin{bmatrix}
    A_{RR} & A_{RS} \\
    0_{|S|\times (n-|S|)} & -I_{|S|\times |S|}
\end{bmatrix}
\begin{bmatrix}
x_R(t) \\
x_S(t)
\end{bmatrix}\notag \\
&  \ + \begin{bmatrix}
0_{(n-|S|)\times |S|} \\
I_{|S|\times |S|} 
\end{bmatrix} u(t), \label{eq:clark}
\end{align}
Moreover, they showed that this problem can be solved in polynomial time by formulating it as a matroid intersection problem
\begin{align*}
    \maximize\{|R| \colon R\in \mathcal{M}^\ast_1 , R\in \mathcal{M}^\ast_2 \},
\end{align*}
where $\mathcal{M}^\ast_1$ and $\mathcal{M}^\ast_2$ denote the duals of the matroid determined by the matrices of descriptor system (\ref{eq:clark}).
This is a special case of MCP1, as formulated in Section~\ref{sec:prb} because the inputs must be connected to state nodes directly. To the best of our knowledge, this is the only work on the minimum input selection for ``structural'' descriptor system (\ref{eq:descriptor}), although the structural controllability analysis without considering input selections can be found in \cite{murota00,yamada1985generic}. We address a more general MCP1 in this paper.

Unlike the problems in \cite{clark2017input} and this study, \cite{zhang2020sparsest} considered the input selection problems for non-structural descriptor system (\ref{eq:descriptor}) formulated as
\begin{align}
& 
\begin{cases}
    \underset{\text{$B\in \Real^{n\times n}$}}{\text{minimize}} & \|B\|_0 \\
    \text{subject to} & \text{system (\ref{eq:descriptor}) is controllable,}
\end{cases} \label{prb:zhang1} \\
&
    \begin{cases}
        \underset{\text{$B\in \Real^{n\times l}$}}{\text{minimize}} & \|B\|_0 \\
        \text{subject to} & \text{system (\ref{eq:descriptor}) is controllable,}
    \end{cases} \label{prb:zhang2}
\end{align}
where $l$ denotes an integer $l\leq n$. 
That is, in contrast to the ``structural'' controllability problem in \cite{clark2017input} and this paper, \cite{zhang2020sparsest} addressed usual controllability problems.
Moreover, the column size of $B$ in Problem~(\ref{prb:zhang2}) is given in advance, whereas MCP0 is the problem of minimizing the column size of $B$.
In \cite{zhang2020sparsest}, the authors showed that Problems~(\ref{prb:zhang1}) and (\ref{prb:zhang2}) are NP-hard. These results are analogous to the NP hardness of MCPs for non-structural LTI systems \cite{olshevsky2015minimal}.

\section{Problem Settings}
\label{sec:prb}
In this section, we formulate novel MCPs for descriptor system (\ref{eq:descriptor}).

Because the controllability of system (\ref{eq:descriptor}) is different from that for system (\ref{eq:lti}), we first define the controllability for descriptor system (\ref{eq:descriptor}) and describe the equivalent conditions.

We assume that system (\ref{eq:descriptor}) is solvable, \textit{i.e.}, for any initial state $x(0)$ with an admissible input $u(t)$, there exists a unique solution $x(t)$ to Eq.~(\ref{eq:descriptor}). This condition is equivalent to 
\begin{align}
    \rank (A-sF) = n, \label{eq:solvability}
\end{align}
where $s$ is an indeterminant \cite{murota00} \cite{yip1981solvability}.

Some definitions of the controllability\cite{yip1981solvability} are provided in the descriptor system theory, as explained in Remark 1.
In this paper, we call system (\ref{eq:descriptor}) controllable if for any admissible initial state $x(0)$ satisfing equation (\ref{eq:descriptor}), there exists an input $u(t)$ and a final time $T\geq 0$ such that $x(T) = 0$.

An algebraic characterization of the controllability is provided below \cite{yip1981solvability, berger2013controllability}:
\begin{proposition}
    \label{prop:scalg}
    Descriptor system (\ref{eq:descriptor}) is controllable \textit{if and only if} 
    \begin{align}
        \rank \left[ A-zF \mid B \right] & = n \quad (z\in \mathbb{C}). \label{eq:rank}
    \end{align}
\end{proposition}
In Proposition~\ref{prop:scalg}, the matrix $\left[ A-sF \mid B \right]$ is called the modal controllability matrix.

System (\ref{eq:descriptor}) is called structurally controllable if condition (\ref{eq:rank}) in Proposition \ref{prop:scalg} holds for (\ref{eq:descriptor}) with generic matrices $F$, $A$, and $B$.
The generic matrices mean that all non-zero entries represent free parameters, which are algebraically independent \cite{murota00}.

We now formulate novel MCPs for descriptor system (\ref{eq:descriptor}).
MCP0 and MCP1 for descriptor system (\ref{eq:descriptor}) can be formalized as
\begin{align*}
    & \text{(MCP0)}\begin{cases}
        \underset{\text{$B\in \mathcal{G}^{n\times m}$}}{\text{minimize}} & m \\
        \text{subject to} & \text{system (\ref{eq:descriptor}) is structurally controllable,}
    \end{cases} \\
    & \text{(MCP1)}\begin{cases}
        \underset{\text{$b\in \mathcal{G}^{n}$}}{\text{minimize}} & \|b\|_0 \\
        \text{subject to} & \text{system (\ref{eq:descriptor}) with $B = \mathrm{diag}(b)$ is} \\ & \text{structurally controllable.}
    \end{cases}
\end{align*}

Although the $i$-th row differential equation corresponds to $x_i$ for LTI system (\ref{eq:lti}), there is no such correspondence because of the presence of $F$ for descriptor system (\ref{eq:descriptor}). Therefore, the directed graph representation of the system in \cite{liu2011controllability} cannot be applied. Thus, we cannot use the proofs of the theorems on MCP0 \cite{liu2011controllability} and MCP1\cite{pequito2015framework} for descriptor system (\ref{eq:descriptor}) with generic matrices.

\begin{remark}
\label{rem:controllability}
Controllability for descriptor systems has some variants, such as complete controllability, R-controllability, impulse controllability, and strong controllability\cite{berger2013controllability}. Our controllability in this paper corresponds to R-controllability. For more detail, see \cite{dai1989singular,berger2013controllability} and references therein.
\end{remark}
\section{Analysis of the MCPs}
\label{sec:analysis}
We analyze the MCPs for descriptor system (\ref{eq:descriptor}) with generic matrices using a graph theoretic approach.

To this end, we first describe the graphical representation of descriptor system (\ref{eq:descriptor}) with generic matrices and graph theoretical controllability.

Although there are several graph representations of descriptor system (\ref{eq:descriptor}) \cite{reinschke1997digraph,yamada1985generic,murota00}, we use a bipartite graph representation and its DM decomposition proposed in \cite{murota00}. The DM decomposition is explained in Appendix~\ref{app:dmdecomp}. This bipartite representation is considerably different from the bipartite graph for LTI system (\ref{eq:lti}) described in Section~\ref{subsec:resultlti}.

The bipartite graph $G = (V^+,V^-; E)$ associated with descriptor system (\ref{eq:descriptor}) is defined as follows: the node sets $V^+ , V^-$ and the edge set $E$ are defined as
\begin{align}
\begin{cases}
   V^+ := X \cup U, \  \ 
   V^- := \{e_1,\dots,e_n\}, \\
   E := E_A \cup E_F \cup E_B 
\end{cases}
\label{eq:defgraph}
\end{align}
with $E_A := \{ ( e_j , x_i ) \mid A_{ij} \neq 0 \}$, $E_F := \{ ( e_j , x_i ) \mid F_{ij} \neq 0 \}$ and $E_B = \{ ( e_j , u_i ) \mid B_{ij} \neq 0 \}$, where $X$ and $U$ are defined as (\ref{eq:nodex}) and (\ref{eq:nodeu}), respectively, and $e_i$ corresponds to an $i$-th row of system (\ref{eq:descriptor}).
An edge belonging to $E_F$ is called an s-arc.
Further, we call a DM s-component for a DM component that has s-arcs.
The symbol $\nu(G)$ denotes the size of a maximum matching of $G$. Let $G_k \ ( k = 0 , \dots , b , \infty )$ be DM components, and $G_k \ ( k = 1,\dots,b )$ are called consistent DM conponents; $G_0$ and $G_\infty$ are called inconsistent DM components.

In \cite{murota00}, we find the following graph theoretic characterization of structural controllability for descriptor system (\ref{eq:descriptor}).
\begin{proposition}
\label{prop:sc}
Descriptor system (\ref{eq:descriptor}) is structurally controllable \textit{if and only if}
\begin{enumerate}
    \item $\nu(G_{A-sF}) = n$
    \item $\nu(G_{\left[A \mid B\right]}) = n$
    \item No consistent DM components contain s-arcs
\end{enumerate}
where $G_{A-sF} = (X,V^-; E_A\cup E_F)$ and $G_{\left[A \mid B\right]} = (X,V^-; E_A\cup E_B)$. 
\end{proposition}

Condition 1) represents the solvability of the system described in Eq.~(\ref{eq:solvability}) in Section~\ref{sec:prb},
and the latter two conditions correspond to (\ref{eq:rank}) in Proposition \ref{prop:scalg}.
From assumption (\ref{eq:solvability}), condition 1) in Proposition~\ref{prop:sc} is always satisfied. This means that there are no inconsistent components $G_0$ and $G_\infty$ of $G_{A-sF}$.

To illustrate 1), 2), and 3) in Proposition \ref{prop:sc}, consider an example of descriptor system (\ref{eq:descriptor}) with
\begin{align}
    F = \begin{bmatrix}
        f_1 & 0 & 0 \\
        f_2 & f_3 & f_4 \\
        0 & 0 & 0
    \end{bmatrix},\ 
    A = \begin{bmatrix}
        a_1 & 0 & 0 \\
        a_3 & 0 & 0 \\
        0 & 0 & a_4 
    \end{bmatrix},\ 
    B = \begin{bmatrix}
        0 \\
        b \\
        0
    \end{bmatrix}. \label{eq:examplesystem}
\end{align}
Then, the corresponding bipartite representation is shown in Fig.~\ref{fig:dmdecomp} (a).
The maximum matching sizes for $G_{A-sF}$ and $G_{[A\mid B]}$ are both 3, which means that conditions 1) and 2) in Proposition~\ref{prop:sc} hold.
The graph $G_0$ is an inconsistent DM component and $G_1,G_2$ are consistent DM components;
further, the partial order $\preceq$ is introduced by $ G_0 \preceq G_1$ and $G_0 \preceq G_2$. In this case, $G_\infty$ does not exist because we can identify a maximum matching of $G$ as a perfect matching $M = \{(x_1,e_1),(x_2,e_2),(x_3,e_3)\}$, that is, $V^\infty = \emptyset$, where $V^\infty$ is defined in Appendix~\ref{app:dmdecomp}. This holds under the solvability (\ref{eq:solvability}) of descriptor system (\ref{eq:descriptor}).
This decomposition corresponds to the block diagonalization of the modal controllability matrix
\begin{align*}
    D(s) := \left[ A-sF \mid B \right] 
    =
    \begin{bmatrix}
        a_1 - sf_1 &  &  &  \\
        a_3 - sf_2 & -sf_3 & -sf_4 & b \\
         &  & a_4 & 
    \end{bmatrix}.
\end{align*}
That is, we can choose permutation matrices $P,Q$ and make $PD(s)Q$ block the upper triangular matrix as
\begin{align}
     PD(s)Q =
     \begin{bmatrix}
        b & -sf_3 & a_2 - s f_2 & -sf_4  \\
        & & a_1-sf_1 & \\
        & & & a_4
    \end{bmatrix}.   \label{eq:dmdecomp}
\end{align}
The diagonal block matrices are given as
\begin{align*}
    D_0(s) = 
    \begin{bmatrix}
        b & -sf_3
    \end{bmatrix}, 
    D_1(s) =
    \begin{bmatrix}
        a_1-sf_1
    \end{bmatrix},  
    D_2(s) = \begin{bmatrix}
        a_4
    \end{bmatrix}
\end{align*}
Block matrices $D_i$ correspond to DM components $G_i$.
As can be seen from Fig.~\ref{fig:dmdecomp} (b) and Eq.~(\ref{eq:dmdecomp}), a consistent component $G_1$ has an s-arc.
Thus, this system is not structurally controllable by condition 3) in Proposition \ref{prop:sc} because $x_1$ is not driven by an input in the 1st equation of the system $f_1 \dot{x}_1(t) = a_1 x_1(t)$.

\begin{figure}
    \centering
    \begin{tabular}{c}
    \begin{minipage}{0.5\hsize}
    \centering
    \includegraphics[width=3.0cm]{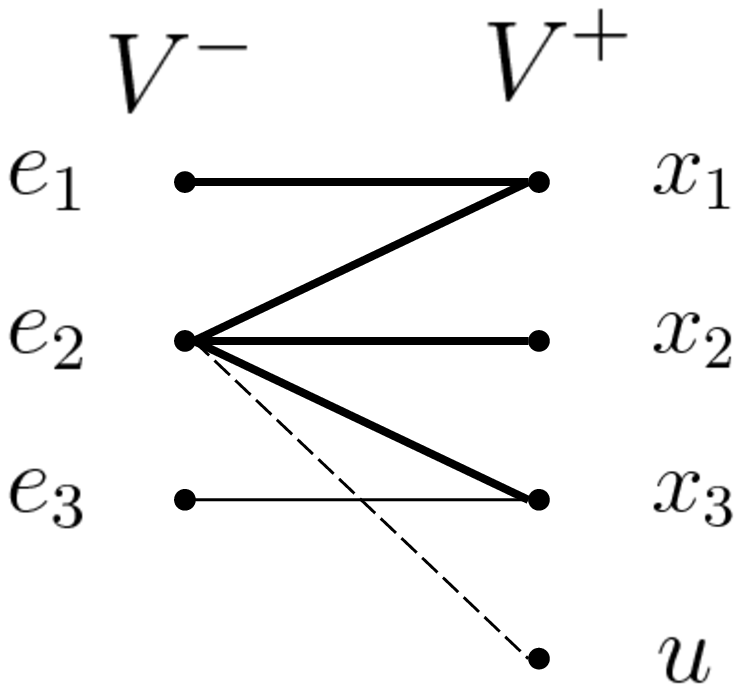}
    \hspace{1.6cm} (a) 
    \end{minipage}
    \begin{minipage}{0.5\hsize}
    \centering
    \includegraphics[width=3.5cm]{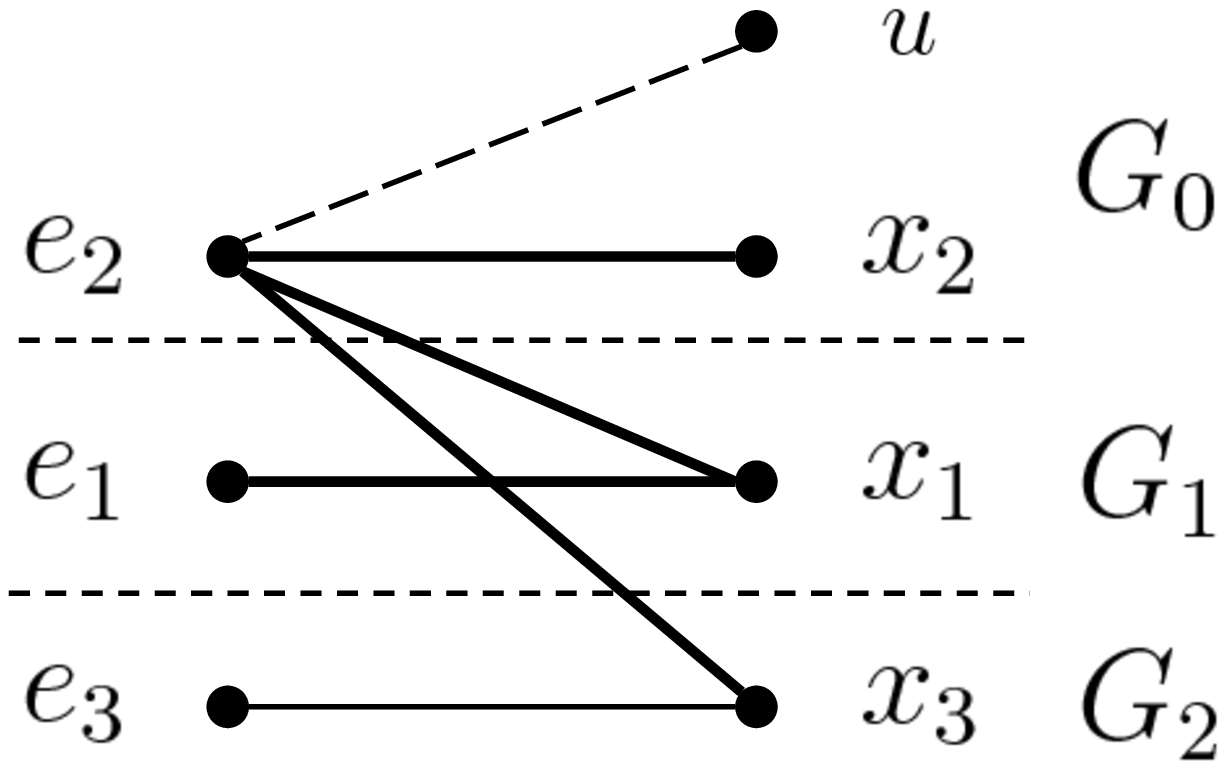}
    \hspace{1.6cm} (b) 
    \end{minipage}
    \end{tabular}
    \caption{(a) Bipartite graph representation of the descriptor system (\ref{eq:examplesystem}) and (b) its DM decomposition. Bold edges represent s-arcs, and dashed edges represent input edges.}
    \label{fig:dmdecomp}
\end{figure}
\subsection{Analysis of the MCP0}
In this subsection, we prove that in the MCP0 case, the minimum number of input nodes for descriptor system (\ref{eq:descriptor}) coincides with that for LTI system (\ref{eq:lti}).
\begin{theorem}[minimum input theorem]
\label{thm:mit}
Suppose that $\nu(G_{A-sF}) = n$, that is, system (\ref{eq:descriptor}) is solvable.
Then, the minimum number of inputs $n_D$ for system (\ref{eq:descriptor}) to be structurally controllable is
\begin{align}
    n_D = \max \{ n - \nu(G_A) , 1 \}, \label{eq:mdns}
\end{align}
where $G_A = (X,V^-;E_A)$.
\end{theorem}
\begin{proof}
If $G_A$ has a perfect matching, that is $\nu(G_A) = n$, then condition 2) in Proposition \ref{prop:sc} is automatically satisfied by using edges of $E_A$ in finding the maximum matching of $G_{\left[A\mid B\right]}$ for any $B$.
Moreover, adding a single input node $u$ and adding edges from $u$ to each s-consistent component, condition 3) in Proposition \ref{prop:sc} is satisfied. Thus, the minimum number of input nodes is $1$.

We assume that $G_A$ does not have perfect matching, that is $\nu(G_A) < n$.
Let $n_D$ be the number of input nodes that the associated system (\ref{eq:descriptor}) is structurally controllable and
\begin{align}
    n_D < n-\nu(G_A). \label{eq:ndineq}
\end{align}
From condition 2) in Proposition \ref{prop:sc}, $\nu(G_{\left[A \mid B\right]}) = n$ holds.
Let $M$ be a maximum matching of $G_{\left[A \mid B\right]}$, and then, let $M\setminus E_B$ be a matching of $G_A$, and
\begin{align}
    |M\setminus E_B| \geq n - n_D. \label{eq:mineq}
\end{align}
From (\ref{eq:ndineq}) and (\ref{eq:mineq}), $|M\setminus E_B| > \nu(G_A)$. This contradicts the maximality of $\nu(G_A)$. Thus, $n_D \geq n - \nu(G_A)$.
Let $M_A$ be the maximum matching of $G_A$ and a number of input nodes be $n-\nu(G_A)$.
Connecting each input node to a node of $V^{-} \setminus \partial^{-} M_A$, $\nu(G_{\left[A \mid B \right]}) = n$ is satisfied, \textit{i.e.}, condition 2) in Proposition~\ref{prop:sc} holds. In addition, by connecting input nodes to s-consistent DM components of $G_{A-sF}$, 3) in Proposition~\ref{prop:sc} is satisfied without increasing the number of input nodes.
The reason is as follows. 
From assumption (\ref{eq:solvability}) and the structure of $G$, we can use the perfect matching of $G_{A-sF}$ as the maximum matching of $G$ in the DM decomposition of $G$. Hence, it follows from (\ref{eq:defgraph}) that $V^+\setminus \partial^+M$ in the construction of the DM decomposition is just $U$. Thus, all s-consistent DM components of $G_{A-sF}$ become inconsistent DM components, and 3) in Proposition~\ref{prop:sc} is satisfied.

Therefore, system (\ref{eq:descriptor}) constructed in the above manner is structurally controllable. \qed
\end{proof}

The statement of Theorem~\ref{thm:mit} for descriptor system (\ref{eq:descriptor}) with generic matrices coincides with that of the minimum input theorem in Section~II in \cite{liu2011controllability} for LTI system (\ref{eq:lti}) with generic matrices.
In \cite{liu2011controllability}, the corresponding proposition was proved using graphical concepts such as cacti. However, we cannot use the concepts for descriptor system (\ref{eq:descriptor}), because (\ref{eq:descriptor}) cannot be represented by the directed graph as mensioned in Section~\ref{sec:prb}.
Instead, we use condition 2) in Proposition~\ref{prop:sc}, which denotes the controllability condition using the maximum bipartite matching.
Moreover, the DM decomposition can be constructed in $O(|E||V|^{1/2})$ time \cite{murota00}. Thus, we can construct $B$ as the solution to MCP0 for descriptor system (\ref{eq:descriptor}) with generic matrices in polynomial time.

To demonstrate how to use this result, we consider autonomous descriptor system (\ref{eq:descriptor}) with $A$ and $F$ defined in (\ref{eq:examplesystem}). Recall that $\nu(G_{A-sF}) = 3$. That is, condition 1) in Proposition~\ref{prop:sc} holds.
Then, the maximum matching of $G_A$ is given as $M = \{a_1=(x_1,e_1),a_4=(x_3,e_3)\}$.
Thus, $\nu(G_A) = 2$ and it follows from (\ref{eq:mdns}) that the minimum node size is $n_D = 1$. From the proof of Theorem~\ref{thm:mit}, we must connect an input node $u$ to the left unmatched node $e_2$. Then, condition 2) in Proposition~\ref{prop:sc} holds.
Moreover, to satisfy condition 3) in Proposition~\ref{prop:sc}, we must add an edge between $u$ and $e_1$ because the DM component $G_2$ of $G_{A-sF}$ that contains $e_1$ has s-arcs (Fig.~\ref{fig:dmdecomp2} (a)). Thus, we have a solution to MCP0 as 
\begin{align*}
    B = \begin{bmatrix}
        b_1 & b_2 & 0
    \end{bmatrix}^\top.
\end{align*}
Then, the DM decomposition of the modal controllability matrix
\begin{align*}
 D(s) = [ A - s F \mid B ] = 
    \begin{bmatrix}
        a_1 - sf_1 &  &  & b_1 \\
        a_3 - sf_2 & -sf_3 & -sf_4 & b_2 \\
         &  & a_4 &
    \end{bmatrix}
\end{align*}
is given as
\begin{align*}
     PD(s)Q =
     \begin{bmatrix}
        b_2 & -sf_3 & a_2 - s f_2 & -sf_4  \\
        b_1 & & a_1-sf_1 & \\
        & & & a_4
    \end{bmatrix},
\end{align*}
where $P$ and $Q$ are permutation matrices.
The corresponding DM decomposition of the bipartite graph is shown in Fig.~\ref{fig:dmdecomp2} (b). This shows that the DM components are $G_0 , G_1$, and the only consistent DM component $G_1$ has no s-arcs, whereas $G_1$ in Fig.~\ref{fig:dmdecomp} has an s-arc.
Thus, from Proposition~\ref{prop:sc}, system (\ref{eq:descriptor}) is structurally controllable.

\begin{figure}
    \centering
    \begin{tabular}{c}
    \begin{minipage}{0.5\hsize}
    \centering
    \includegraphics[width=3.0cm]{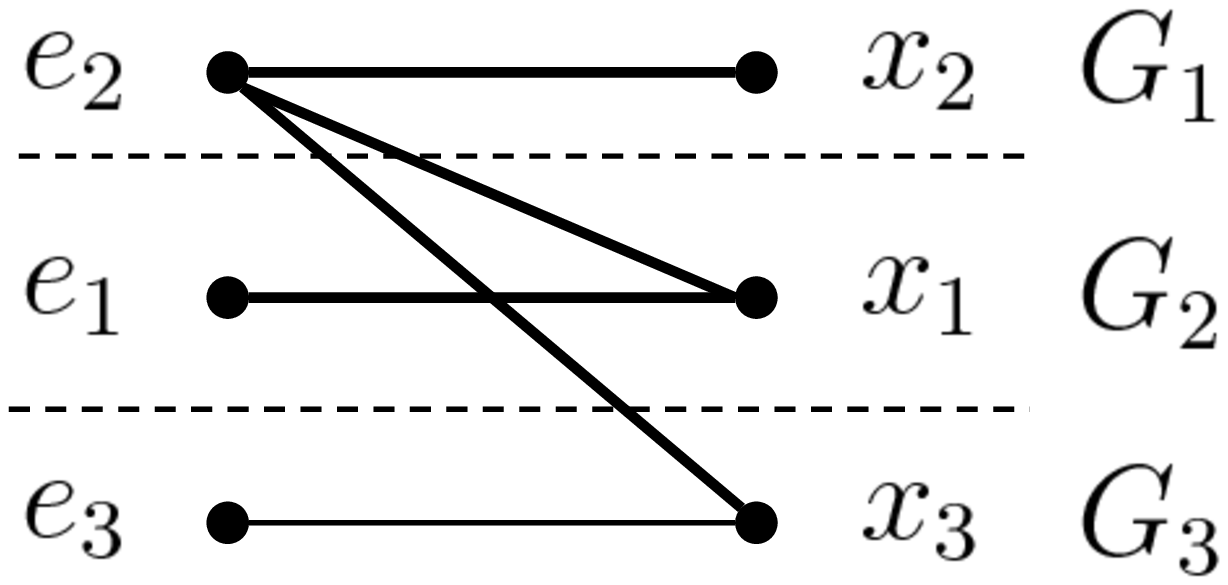}
    \hspace{1.6cm} (a) 
    \end{minipage}
    \begin{minipage}{0.5\hsize}
    \centering
    \includegraphics[width=3.0cm]{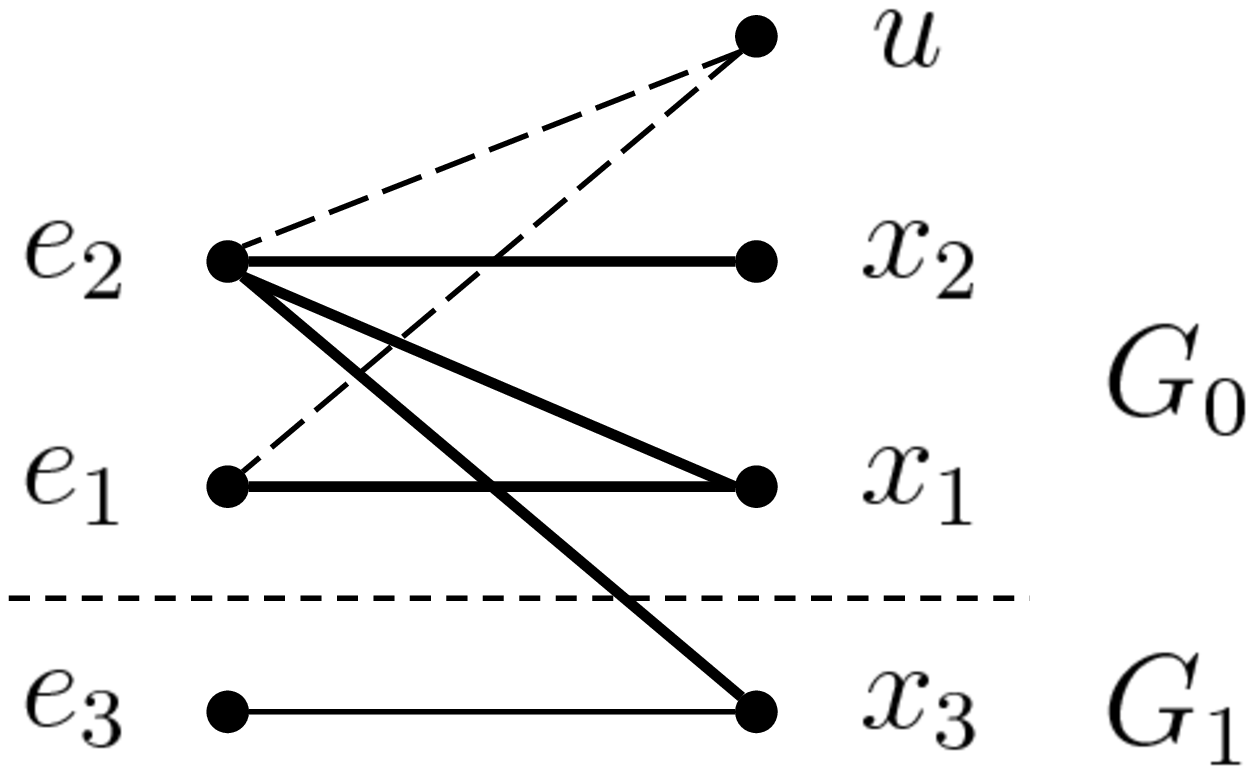}
    \hspace{1.6cm} (b) 
    \end{minipage}
    \end{tabular}
    \caption{(a) DM decomposition of $G_{A-sF}$ and (b) DM decomposition of the descriptor system configured as the solution to MCP0. Bold edges represent s-arcs, and dashed edges represent input edges.}
    \label{fig:dmdecomp2}
\end{figure}

\subsection{Intractability of MCP1 for descriptor systems}
We prove that the set cover problem, which is an NP-hard problem, can be reduced to MCP1 for an instance of descriptor system (\ref{eq:descriptor}) with generic matrices.
This means that MCP1 for descriptor system (\ref{eq:descriptor}) with generic matrices is NP-hard.

To this end, we first define additional terms of the DM components of descriptor system (\ref{eq:descriptor}) and discuss permissible inputs under the assumption of MCP1 that only one state node can be connected to an edge from each input, as illustrated in the right of Fig.~\ref{fig:comparisonMCP01}.

The maximal consistent DM s-component is a consistent DM s-component $G_k$ of $G_{A-sF}$ such that no other DM s-components are greater than $G_k$ related to the partial order $\preceq$. For example, $G_1$ in Fig.~\ref{fig:dmdecomp} (b) is a maximal consistent DM s-component.

A subset $S \subseteq V^-$ is called an \textit{input configuration} if system (\ref{eq:descriptor}) becomes structurally controllable by connecting the inputs to the nodes specified by $S$. An input configuration $S$ is a feasible solution to MCP1 for descriptor system (\ref{eq:descriptor}).
A similar yet slightly different concept of the input configuration can be found in \cite{pequito2015framework} for LTI system (\ref{eq:lti}) with generic matrices.

From the graphical condition of structural controllability in Proposition \ref{prop:sc}, we obtain the following theorem on the input configuration. This theorem clarifies the structure of the input configuration in MCP1 and reduces the set cover problem to MCP1 for an instance of system (\ref{eq:descriptor}).
\begin{theorem}
\label{thm:inputconfig}
    A subset $S \subseteq V^-$ is an input configuration \textit{if and only if} 
    \begin{align*}
        S \supseteq U_M \cup C,
    \end{align*}
    where $U_M$ is the set of left unmatched nodes of a maximum matching $M$ of $G_A$, \textit{i.e.}, $U_M = V^- \setminus \partial^- M$. The symbol $C$ is a set whose nodes can reach all maximal s-consistent components. That is, for any maximal s-consistent component $G_i$, there exists a node $c\in C$ such that $[c] \preceq G_i$. Here, $[c]$ denotes the DM component of $G$ to which $c$ belongs. 
\end{theorem}
\begin{proof}
    Let $B_S$ be a generic diagonal matrix corresponding to the set $S$.
    
    ($\Rightarrow$)
    From the definition of $S$, the system (\ref{eq:descriptor}) with $B=B_S$ is structurally controllable. Then, $\nu(G_{[A\mid B_S]}) = n$ holds from condition 2) in Proposition~\ref{prop:sc}.
    If $\nu(G_A) < n$, there exists a maximum matching $M$ of $G_A$ and $V^- \setminus \partial^- M$ is unmatched nodes. This implies that $S$ must contain $V^-\setminus \partial^-M$, and thus $S \supseteq U_M$.
    
    To show that $S$ must contain $C$, we provide the relationship between the DM components of $G$ and $G_{A-sF}$, where $G$ and $G_{A-sF}$ are defined in the beginning of Section~\ref{sec:analysis}. Let $\{G_i=(V_i^+,V_i^-;E_i)\}$ be the DM components of $G_{A-sF} = (X,V^-;E_A\cup E_F)$. From assumption (\ref{eq:solvability}), condition 1) in Proposition~\ref{prop:sc} holds. Thus, there are no inconsistent components $G_0,G_\infty$.
    Then, the nodes sets $V_0',V_1',V_2',\dots,V_l'$ of the DM components of $G$ are
    \begin{align*}
        \{V_t'\}  & := \{V_k^-\cup V_k^+ \mid e_i \not \in V_k^- \cap S, i = 1,\dots,n \},  \\
        V_0' & := \bigcup_{i=1}^n \{V_k^- \cup V_k^+ \mid e_i \in V_k^- \cap S \} \cup \{ u_1 , \dots , u_{|S|} \}.
    \end{align*}
    
    That is, the DM component that contains the nodes of $S$ will be merged with the inconsistent DM component $G_0'$, and the other components remain consistent (Fig.~\ref{fig:addingB}).
    For $G_i$ to be merged to an inconsistent component $G_0'$ in $G$, an input node must exist and be connected to $G_i$ that is either $G_i$ or a larger (in the sense of the partial order of DM decomposition) consistent component.
    Combining this with condition 3) in Proposition~\ref{prop:sc}, we have $S\supseteq C$.
    
    \begin{figure}
        \centering
        \includegraphics[width=5cm]{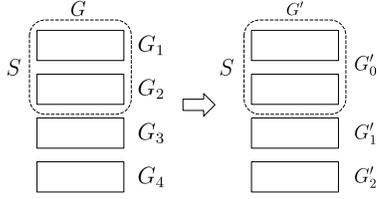}
        \caption{DM components of $G$ and $G'$.}
        \label{fig:addingB}
    \end{figure}
    ($\Leftarrow$)
    Let $S$ contain $U_M \cup C$, and let $M$ be a maximum matching of $G_A$.
    Each input node specified by $S$ connects to the left-unmatched nodes of $G_A$ in the maximum matching $M$. If we choose edges from $A$ and $B_S$ in finding a maximum matching of $G_{[A\mid B_S]}$, $\nu(G_{[A\mid B_S]}) = n$ holds. That is, condition 2) in Proposition~\ref{prop:sc} is satisfied.
    Moreover, from the definition of $C$, all consistent s-components of $G_{A-sF}$ become inconsistent components in the DM decomposition of $G$. This means that no consistent components contain s-arcs; thus, descriptor system (\ref{eq:descriptor}) with $B = B_S$ is structurally controllable. That is, $S$ is an input configuration.
    \qed
\end{proof}

We briefly discuss the set cover problem to prove the intractability of MCP1 for descriptor system (\ref{eq:descriptor}) with generic matrices.
Consider $W := \{w_1,\dots,w_n\}$, and $S_1,\dots,S_k \subseteq W$ with $\bigcup_{k} S_k = W$.
The set cover problem is a problem that finds the minimum size index $I\subseteq \{1,\dots,k\}$ such that $\{S_i\}_{i\in I}$ covers $W$.
\begin{align}
\label{prb:setcover}
    \begin{cases}
        \mathrm{minimize} & \quad |I| \\
        \mathrm{subject\ to} & \quad \bigcup_{i\in I} S_i = W,\quad I \subseteq \{1,\dots,k\}.
    \end{cases}
\end{align}
The set cover problem is NP-hard as shown in Theorem~15.24 in \cite{korte2012combinatorial}.

The computational difficulty of MCP1 for descriptor system (\ref{eq:descriptor}) with generic matrices can be proven 
using Theorem~\ref{thm:inputconfig}.
\begin{theorem}[Intractability of MCP1 for descriptor system]
\label{thm:mcp1}
    MCP1 for descriptor system (\ref{eq:descriptor}) with generic matrices includes the set cover problem. That is, its computational complexity is NP-hard.
\end{theorem}
\begin{proof}
We show that given $W = \{w_1,\dots,w_n\}$, $S_1,\dots,S_k \subseteq W$, we can define a specific autonomous descriptor system (\ref{eq:descriptor}) with $B=0$ corresponding to the set cover problem.

To this end, we first define a bipartite representation of the system. Its DM components correspond to the elements of $W$ and $S_1,\dots,S_k$. Moreover, the order between a DM component associated with an element $w_i$ of $W$ and a subset $S_j$ can be regarded as the inclusion between $w_i$ and $S_j$.
The nodes sets $V^+$ and $V^-$ are defined as $V^+ = \{w^+_1,\dots,w^+_n\} \cup \{s^+_1,\dots,s^+_k\}, V^- = \{w^-_1,\dots,w^-_n\} \cup \{s^-_1,\dots,s^-_k\}$.
Further, we define $E_A = \{(s^-_i,s^+_i)\mid i=1,\dots,k\} \cup \{(w^-_i,w^+_i)\mid i=1,\dots,n\}$ and $E_F = \{(w^-_i,s^+_j) \mid w_i \in S_j \} \cup \{(w^-_i,w^+_i)\mid i=1,\dots,n\}$.
Then, the DM components of this bipartite representation are the subgraph induced by $\{w^+_i,w^-_i\} \ (i=1,\dots,n)$ or $\{s^+_j,s^-_j\} \ (j=1,\dots,k)$. The DM components induced by $\{w^+_i,w^-_i\}$ and $\{s^+_j,s^-_j\}$ correspond to the element $w_i$ of $W$ and the subset $S_j$, respectively.
Moreover, if $w_i \in S_j$, the DM component induced by $\{w^+_i,w^-_i\}$ is less than the DM component induced by $\{s^+_j,s^-_j\}$ in the sense of the order of DM decomposition.
It follows from (\ref{eq:defgraph}) that the bipartite graph $(V^+,V^-;E_A\cup E_F)$ defines an instance of descriptor system (\ref{eq:descriptor}) with $B=0$.

Theorem \ref{thm:inputconfig} implies that the number of input configuration $S$ is 
\begin{align}
    |S| \geq |U_M \cup C|. \label{eq:configineq}
\end{align}
In this case, because we can take only a perfect matching from $E_A$ as a maximum matching $M$, $U_M$ is fixed to $\emptyset$. Combining this with (\ref{eq:configineq}), we have $|S| \geq |C|$.
Moreover, in the MCP1 case, we select $S$ as $C$ because we must select the minimum size of $S$, which corresponds to a feasible solution to MCP1.
Accordingly, MCP1 is equivalent to the following problem:
\begin{align}
\begin{cases}
    \text{minimize}  &  |C|\\
    \text{subject to} & \text{For all s-consistent component $G_k$, } \\
                      & \text{there exists $c\in C$ such that $[c] \preceq G_k$, and} \\
                      & C \subseteq V^-.
\end{cases} \label{prb:setcoverdescriptor} 
\end{align}
From the construction of $E_F$, the s-consistent component corresponds exactly to the element of $W$, and $C$ is chosen from any consistent components containing $s$ (Fig.~\ref{fig:setcover}).
Thus, problem (\ref{prb:setcoverdescriptor}) is equivalent to problem (\ref{prb:setcover}). \qed
\end{proof}

We explain Theorem~\ref{thm:mcp1} using Fig.~\ref{fig:comparisonMCP1}.
In MCP1 for LTI system (\ref{eq:lti}), inputs must be connected to root SCCs.
If a graph representation of system (\ref{eq:lti}) is illustrated in the left of Fig.~\ref{fig:comparisonMCP1}, then we must add two inputs.
However, in MCP1 for descriptor system (\ref{eq:descriptor}), if the maximal DM component contains no s-arc, and therefore, it must not be connected from inputs. Thus, we have to choose minimum inputs that cover all DM components that include s-arcs. Thus, the problem is more complicated.
In \cite{clark2017input}, system (\ref{eq:descriptor}) was restricted to the form given in Eq.~(\ref{eq:clark}) and a polynomial-time solution with a matroid intersection was proposed; however, our result shows that it is difficult to solve in polynomial time for a more general case.

Consider the following instance of the set cover problem: $W = \{w_1,w_2,w_3,w_4\}$ and $S_1 = \{w_1,w_2,w_4\} , S_2 = \{w_1,w_3\} , S_3 = \{w_3,w_4\}$. Then, the solution to the set cover problem (\ref{prb:setcover}) is $ I = \{1,3\}$ because $S_1 \cup S_3 = W$, and we cannot select a better $I$.
Moreover, descriptor system (\ref{eq:descriptor}) defined in the proof of Theorem~\ref{thm:mcp1} is 
\begin{align*}
    A = &
    \begin{blockarray}{*{7}{c} l}
    \begin{block}{*{7}{>{$\footnotesize}c<{$}} l}
      $w_1^+$ & $w_2^+$ & $w_3^+$ & $w_4^+$ & $s_1^+$ & $s_2^+$ & $s_3^+$ \\
    \end{block}
    \begin{block}{[*{7}{c}]>{$\footnotesize}l<{$}}
      a_1 & &  &  &  &  &  & $w_1^-$ \\
      & a_2 &  &  &  &  &  & $w_2^-$ \\
      &  & a_3 &  &  &  &  & $w_3^-$ \\
      &  &  & a_4 &  &  &  & $w_4^-$ \\
      &  &  &  & a_5 &  &  & $s_1^-$ \\
      &  &  &  &  & a_6 &  & $s_2^-$ \\
      &  &  &  &  &  & a_7 & $s_3^-$ \\
    \end{block}
  \end{blockarray}, \\
  F = &
    \begin{blockarray}{*{7}{c} l}
    \begin{block}{*{7}{>{$\footnotesize}c<{$}} l}
      $w_1^+$ & $w_2^+$ & $w_3^+$ & $w_4^+$ & $s_1^+$ & $s_2^+$ & $s_3^+$ \\
    \end{block}
    \begin{block}{[*{7}{c}]>{$\footnotesize}l<{$}}
    f_1 &  &  &  &  f_5 & f_8 &  & $w_1^-$ \\
      & f_2 &  &  & f_6 &  &  & $w_2^-$ \\
      &  & f_3 &  &  & f_9 & f_{10} & $w_3^-$ \\
      &  &  & f_4 & f_7 &  & f_{11} & $w_4^-$ \\
      &  &  &  &  &  &  & $s_1^-$ \\
      &  &  &  &  &  &  & $s_2^-$ \\
      &  &  &  &  &  &  & $s_3^-$ \\
    \end{block}
  \end{blockarray}.
\end{align*}
The bipartite graph representation of this system is shown in Fig.~\ref{fig:setcover}, whose nodes sets $V^+$ and $V^-$ are defined as $V^+ := \{s^+_1,s^+_2,s^+_3\} \cup \{w^+_1,\dots,w^+_4\}$ and $V^- := \{s^-_1,s^-_2,s^-_3\} \cup \{w^-_1,\dots,w^-_4\}$. The DM components $G_4,G_5,G_6,G_7$ and $G_1,G_2,G_3$ correspond to elements of $W = \{w_1,w_2,w_3,w_4\}$ and subsets $\{S_1,S_2,S_3\}$, respectively.
If we choose $S = \{s_1^-,s_3^-\}\subseteq V^-$, $S$ is an input configuration of the system because all s-consistent components $\{G_4,G_5,G_6,G_7\}$ can be reached from nodes of $S$.

\begin{figure}[t]
    \centering
    \includegraphics[width=5.0cm]{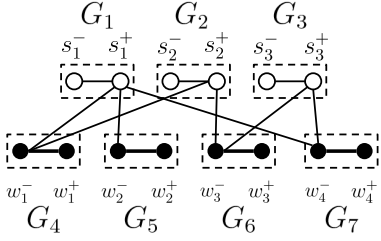}
    \caption{Translation from a set cover problem ($W = \{w_1,w_2,w_3,w_4\}$ and $S_1 = \{w_1,w_2,w_4\}, S_2 = \{w_1,w_3\}$, and $S_3 = \{w_3,w_4\}$) to MCP1 for an instance of descriptor system. Bold edges represent s-arcs. $G_4,\dots,G_7$ are the s-consistent components.}
    \label{fig:setcover}
\end{figure}

\begin{figure}[t]
    \centering
    \includegraphics[width=7cm]{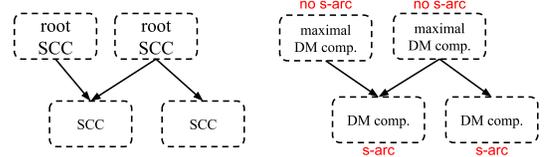}
    \caption{Comparison of MCP1 for LTI system (\ref{eq:lti}) (left) and descriptor system (\ref{eq:descriptor}) (right).}
    \label{fig:comparisonMCP1}
\end{figure}

\section{Conclusion} \label{sec:conclusion}

In this study, we extended the results of MCPs for LTI systems with generic matrices to descriptor systems using bipartite graph representation and DM decomposition.
We showed that the solution to MCP0 for the descriptor system was the same as MCP0 for the LTI system. However, the derivation is considerably different from that in \cite{liu2011controllability} using graph concepts such as cacti.
In fact, we used bipartite graph representation and DM decomposition for the descriptor system. This revealed that MCP0 for descriptor systems can also be computed in polynomial time.
In addition, we found that MCP1 for the descriptor system is NP-hard because the set cover problem can be reduced to MCP1 for an instance of the descriptor systems. This result is completely different from the MCP1 for LTI systems, which can be solved in polynomial time. 

In the case of MCPs for LTI systems, dealing with the system as a structural system allows MCPs, which were originally NP-hard problems \cite{olshevsky2015minimal}, to be solved in polynomial time \cite{liu2011controllability,pequito2015framework,olshevsky2015minimum}.
However, our result on MCP1 for descriptor systems showed that dealing with the systems as a structural system does not make the problem easier at all. Therefore, we need to limit the scope of the target system, similar to that in \cite{clark2017input}, to develop an exact algorithm for solving the MCP1.
In addition, an approximate algorithm for solving MCP1 for structural descriptor systems should be considered in future work.

In physical systems, the solutions to MCP0 and MCP1 for descriptor systems may not be applicable for descriptor systems. For instance, if an equation in the descriptor system represents a physical constraint, then there would be no physical meaning to adding external input to it. To avoid this, the introduction of forbidden nodes as constraints to MCP0 and MCP1, would be useful, as in \cite{olshevsky2015minimum}.

Furthermore, structural controllability requires that the system parameters be algebraically independent. This means that all non-zero system parameters are free, which may be a strong assumption for practical situations. To avoid this assumption, strong structural controllability has been proposed in \cite{mayeda1979strong}, and the input selection problem on strong structural controllability for descriptor systems has been studied in \cite{popli2019selective}. Further, a graph-theoretic analysis would be one of the future works for strong structural controllability.
\appendix

\subsection{DM decomposition}
\label{app:dmdecomp}
We introduce the algorithm of the DM decomposition for a bipartite graph $G = (V^+,V^-;E)$, which is mathematically derived from the distributive lattice determined by the minimizer of some submodular function (see \cite{murota00} for details).
We define $M$ as a maximum matching of $G$ and an auxiliary directed graph $\tilde{G}_M$ whose edges are oriented from $V^+$ to $V^-$ except for $M$.
Then, the DM decomposition algorithm is given as (Fig.~\ref{fig:dmdecompalgorithm}):
\begin{enumerate}
    \item Find a maximum matching $M$ on $G$.
    \item Construct an auxiliary directed graph $\tilde{G}_M$.
    \item $V_0 := \{ v\ \in V^+\cup V^- \mid \exists u\in V^+\setminus \partial^+ M \ u \to_{\tilde{G}_M} v \}$, where $u \to_{\tilde{G}_M} v$ indicates the existence of a directed path from $u$ to $v$ on $\tilde{G}_M$.
    \item $V_\infty := \{ v\ \in V^+\cup V^- \mid \exists u\in  V^-\setminus \partial^- M \ v\to_{\tilde{G}_M} u \}$.
    \item $G' := \tilde{G}_M \setminus (V_0\cup V_\infty)$.
    \item Let $V_k\ (k=1,\dots,b)$ be the strong components of $G'$ and undirected graph $G_k\ (k=0,1,\dots,b,\infty)$ be the subgraph of $G$ induced on $V_k\ (k=0,1,\dots,b,\infty)$. Note that each $G_k$ is an undirected graph.
\end{enumerate}
The order $G_i \preceq G_j$ for the consistent components $G_i$ and $G_j$ is defined as
\begin{center}
     \text{``There is a directed path from $G_j$ to $G_i$ in $\tilde{G}_M$,"}
\end{center}
and the order between the consistent component $G_i$ and inconsistent components $G_0$ and $G_\infty$ are defined as $G_0 \preceq G_i$ and $G_i \preceq G_\infty$, respectively. Then, $\preceq$ is a partial order.
Moreover, the decomposition constructed by the algorithm does not depend on an initially chosen maximum matching $M$ of $G$ as step 1), as shown in Lemma~2.3.35 in \cite{murota00}.

The most computational bottleneck in the construction of the DM decomposition is to find the maximum matching of $G$. This can be achieved in $O(|E||V|^{1/2})$ by using the augmentation path algorithm \cite{korte2012combinatorial}. Thus, the computational complexity of DM decomposition is $O(|E||V|^{1/2})$.

\begin{figure}
    \centering
    \includegraphics[width=8.3cm]{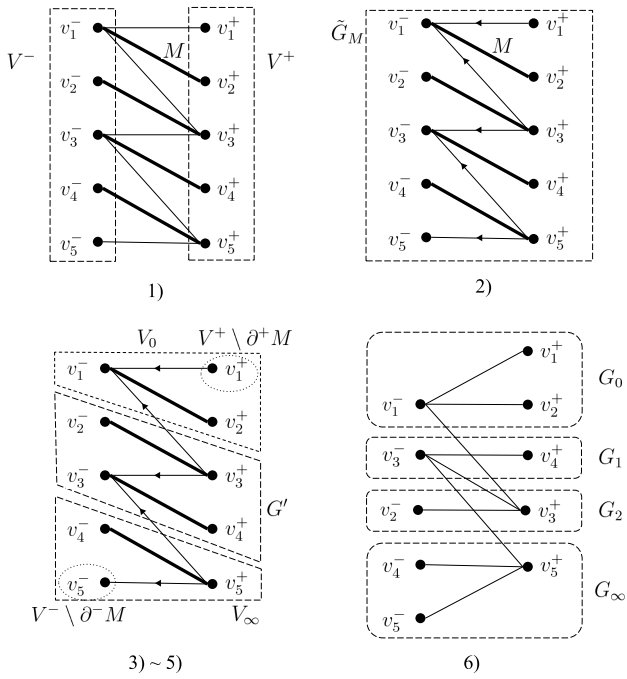}
    \caption{Construction of DM decomposition. Bold edges represent maximum matching.}
    \label{fig:dmdecompalgorithm}
\end{figure}

%


\section*{Acknowledgment}
This work was supported by Japan Society for the Promotion of Science KAKENHI under Grant 20K14760. 

\ifCLASSOPTIONcaptionsoff
  \newpage
\fi



\bibliographystyle{IEEEtran}
\bibliography{main.bib}




%




\end{document}